\documentclass[reqno,A4paper]{amsart}
\usepackage{xcolor}
\usepackage{tikz-cd}
\usepackage[all]{xy}
\usepackage{hyperref}
\usepackage{amsmath}
\usepackage{amssymb}
\usepackage{amsthm}
\usepackage{booktabs}
\usepackage{enumerate}
\usepackage{mathrsfs} 
\usepackage{eqlist}
\usepackage{array}

\theoremstyle{plain}
\newtheorem{theorem}{Theorem}[section]
\newtheorem{proposition}[theorem]{Proposition}
\newtheorem{lemma}{Lemma}[section]
\newtheorem{corollary}{Corollary}[theorem]

\theoremstyle{definition}
\newtheorem{definition}{Definition}
\newtheorem{remark}{\textup{Remark}} 
\newtheorem{example}{\textit{Example}} 

\numberwithin{equation}{section}

\begin{document}

\title[A new categorical equivalence for Stone Algebras]
{A new categorical equivalence for Stone Algebras}
\author[Ismael Calomino \and Gustavo Pelaitay]
{Ismael Calomino* \and Gustavo Pelaitay**}

\newcommand{\acr}{\newline\indent}

\address{\llap{*\,}Facultad de Ciencias y Artes\acr
                   Universidad Cat\'olica de \'Avila\acr
                   C/Canteros s/n 05005 \'Avila\acr
                   ESPAÑA}
\email{ismaelm.calomino@ucavila.es}

\address{\llap{**\,}CONICET and Instituto de Ciencias B\'asicas\acr
                    Universidad Nacional de San Juan\acr
                    San Juan 5400\acr
                    ARGENTINA}
\email{gpelaitay@gmail.com}




\begin{abstract}The aim of this paper is to give a categorical equivalence for Stone algebras. We introduce the variety of Stone-Kleene algebras with intuitionistic negation, or Stone KAN-algebras for short, and explore Kalman's construction for Stone algebras. We examine the centered algebras within this new variety and prove that the category of Stone algebras is equivalent to the category of centered Stone KAN-algebras. Moreover, inspired by Monteiro's construction for Nelson algebras, we propose a method to construct a centered Stone KAN-algebra from a given Stone KAN-algebra and show the connection between Kalman's construction and Monteiro's construction.  
\end{abstract}

\maketitle

\section{Introduction} \label{sec1}

The definition of a Nelson algebra was introduced by Rasiowa in \cite{Rasiowa} and plays a role in the study of constructive propositional calculus with strong negation, analogous to the role of Boolean algebras in classical propositional calculus. Nelson algebras can be viewed as Kleene algebras with an additional binary operation which generalizes relative pseudocomplementation (\cite{Cignoli,J14}). In \cite{Sendlewski}, Sendlewski generalized results from Fidel and Vakarelov \cite{Fidel,Vakarelov} to prove the deep connection between Nelson algebras and Heyting algebras. To establish a similar relationship for a broader class of algebras, Sendlewski introduced the class $ \mathcal{K}_{\omega} $ of Kleene algebras with weak pseudocomplementation in \cite{Sendlewski1}, aiming to establish a topological relationship between this class and pseudocomplemented distributive lattices. The class $ \mathcal{K}_{\omega} $ encompasses the implication-free reducts of Nelson algebras, with the additional operation interpreted as weak pseudocomplementation. In \cite{Conrado-Miguel-Hernan}, the authors examine the connection between strong negation and intuitionistic negation in Nelson algebras. They introduce the variety of Kleene algebras with intuitionistic negation (abbreviated as KAN-algebras) and apply Kalman’s construction (see \cite{Kalman}) for pseudocomplemented distributive lattices. Through this construction, they establish that the algebraic category \textbf{PDL} of pseudocomplemented distributive lattices is equivalent to the algebraic category \textbf{KANc} of centered KAN-algebras. Furthermore, they show in \cite{Conrado-Miguel-Hernan} that the variety of KAN-algebras coincide with the class of $\mathcal{K}_\omega$-algebras.

On the other hand, it is well known that a Stone algebra is a pseudocomplemented distributive lattice $A$ which satisfies the equation $x^{\ast} \vee x^{\ast \ast} = 1$, where $^{*}$ denotes the pseudocomplement of $A$. Stone algebras were introduced by Grätzer and Schmidt in response to a problem given by Stone. These algebras are characterized as bounded distributive lattices in which each prime ideal contains a unique minimal prime ideal (\cite{GS}). Since their introduction, Stone algebras have been extensively studied in the literature (\cite{Balbes,G71}) and have found applications in various areas, such as conditional event algebras and rough sets (\cite{PP,Walker}). 

A central motivation for our work was to identify a suitable category that completes the following diagram:

\vspace{0.25cm}
\begin{center}
\[
\xymatrix{
\mathbf{Stone} \ar[d] \ar[r]^{K} & \mathbf{?} \ar[d]  \\
\mathbf{PDL} ~\ar[r]_{K} & \mathbf{KANc} }
\]
\end{center}
\begin{center}  Fig.1
\end{center}
\vspace{0.25cm}
where \textbf{Stone} is the algebraic category of Stone algebras, ${\bf{Stone}} \to {\bf{PDL}}$ represents the obvious forgetful functor and $K$ is the Kalman’s functor. To address this challenge, we introduce the variety of {\it centered Stone KAN-algebras} and prove that the category {\bf{Stone}} is equivalent to the category of centered Stone KAN-algebras. To this end, we leverage Kalman's construction described in \cite{Conrado-Miguel-Hernan} for pseudocomplemented distributive lattices. In forming this categorical equivalence, we take a distinct approach compared to that of Cignoli in \cite{Cignoli}, who uses topological techniques to prove that the \textbf{Stone} category is equivalent to the category of centered regular $\alpha$-De Morgan algebras.

A second motivation was to identify a functor that would ensure the commutativity of the following diagrams:

\vspace{0.25cm}
\begin{center}
\[
\xymatrix{
\mathbf{KAN} \ar[dr]_{?} \ar[r]^{\mathbf{\theta}} & \mathbf{PDL} \ar[d]^{K}  \\
& \mathbf{KANc} 
}
\quad\quad
\xymatrix{
\mathbf{SKAN} \ar[dr]_{?} \ar[r]^{\mathbf{\theta}} & \mathbf{Stone} \ar[d]^{K}  \\
& \mathbf{SKANc} 
}
\]
\end{center}
\vspace{0.25cm}
where $\theta$ is the functor induced by a particular congruence (see \cite{Conrado-Miguel-Hernan}). To achieve this goal, we draw inspiration from Monteiro's construction for centered Nelson algebras to propose a method that constructs a centered KAN-algebra (centered Stone KAN-algebra) from a given KAN-algebra (Stone KAN-algebra). This construction forms the basis for the desired functor, which we denote as {M} in reference of Monteiro’s approach.

The categories studied in this work are summarized below. This summary is provided to enhance the clarity and overall presentation of the work.

\vspace{0.25cm}
\begin{center}
\begin{table}[h!]
\centering
\renewcommand{\arraystretch}{1.2}
\setlength{\tabcolsep}{8pt}
\begin{tabular}{@{}>{\bfseries}l l l@{}}
\toprule
\textbf{Category} & \textbf{Objects} & \textbf{Morphisms} \\ 
\midrule
$\mathbf{PDL}$     & Distributive $p$-algebras         & Algebra homomorphisms \\ 
$\mathbf{Stone}$   & Stone algebras                  & Algebra homomorphisms \\ 
$\mathbf{KAN}$     & KAN-algebras                     & Algebra homomorphisms \\ 
$\mathbf{KANc}$    & Centered KAN-algebras           & Algebra homomorphisms \\ 
$\mathbf{SKAN}$    & Stone KAN-algebras              & Algebra homomorphisms \\ 
$\mathbf{SKANc}$   & Centered Stone KAN-algebras     & Algebra homomorphisms \\ 
\bottomrule
\end{tabular}
\caption{Categories with their corresponding objects and morphisms.}
\label{tab:categories}
\end{table}
\end{center}

\section{Preliminaries} \label{sec2}

In this section we review key results and definitions from \cite{Conrado-Miguel-Hernan}, focusing on the class of centered KAN-algebras. We also establish several results that will be instrumental throughout the remainder of this paper. It is worth mentioning that throughout this paper, if $\bf{V}$ is a variety, we will denote its corresponding algebraic category by the same symbol.

\subsection{Distributive p-algebras}

A pseudocomplemented distributive lattice, or distributive p-algebra for short, is an algebra $\langle A, \vee, \wedge, ^{\ast}, 0, 1 \rangle$ of type $(2,2,1,0,0)$, where $\langle A, \vee, \wedge, 0, 1 \rangle$ is a bounded distributive lattice, and for each $a, b \in A$, the following condition holds: $a \wedge b = 0$ if and only if $a \leq b^{\ast}$. So, $a^{\ast}$ is the greatest element of $A$ that is disjoint from $a$, that is, $a^{\ast} = \max \{x \in A \colon a \wedge x = 0\}$. In every distributive p-algebra, the conditions $1 = 0^{\ast}$ and $0 = 1^{\ast}$ necessarily hold (\cite{Balbes}). A distributive p-algebra $A$ is called a Stone algebra if it satisfies the equation 
\begin{equation} \label{stone}
x^{\ast} \vee x^{\ast\ast} = 1  
\end{equation}
for all $x \in A$. The class of Stone algebras is a variety introduced by Grätzer and Schmidt in \cite{GS}, named in honor of M. H. Stone. These algebras have been extensively studied by various authors and continue to be a significant area of research (see \cite{Chen1,Chen2,Chen3,Katrinak1,Katrinak2}). Moreover, they play a crucial role in the theory of rough sets (\cite{D,J14,Kumar2024,PP,GeWal}).

\subsection{KAN-algebras}

Recall that a Kleene algebra is an algebra $\langle T, \vee, \wedge, \sim, 0, 1 \rangle$ of type $(2,2,1,0,0)$ where $\langle T, \vee, \wedge, 0, 1 \rangle$ is a bounded distributive lattice and
\begin{enumerate}
\item[(K1)] $\sim \sim x = x$,
\item[(K2)] $\sim (x \vee y) = \sim x \wedge \sim y$,
\item[(K3)] $x \wedge \sim x \leq y \vee \sim y$.
\end{enumerate}
Additionally, in a Kleene algebra $T$, the following properties hold:
\begin{enumerate}
\item $\sim 0 = 1$ and $\sim 1=0$,
\item $\sim (x \wedge y) = \sim x \vee \sim y$,
\item $x \leq y$ implies $\sim y \leq \sim x$.
\end{enumerate}
A Kleene algebra is called centered if it has a center, i.e., an element $c$ such that $\sim c = c$. In every Kleene algebra the center, when it exists, is unique. 

Recall that a Nelson algebra (\cite{Rasiowa, Sendlewski}) is a Kleene algebra $T$ such that for each pair $x,y \in T$, there exists a binary operation $\to$ given by $x \to y := x \to_{\mathsf{HA}} (\sim x \vee y)$ (where $\to_{\mathsf{HA}}$ is the Heyting implication), and for every $x, y, z$, it holds that $(x \wedge y) \to z = x \to (y \to z)$. Nelson algebras can be viewed as algebras $\langle T, \vee, \wedge, \to, \sim, 0, 1 \rangle$ of type $(2, 2, 2, 1, 0, 0)$. The class of Nelson algebras forms a variety (\cite{BrignoleMonteiro,Brignole}).

Now, we introduce the class of Kleene algebras with intuitionistic negation, or KAN-algebras. The KAN-algebras have a weaker structure than Nelson algebra, just like distributive p-algebras have a weaker structure than Heyting algebras.

\begin{definition}\cite{Conrado-Miguel-Hernan} \label{defKAN}
A {\it{Kleene algebra with intuitionistic negation}}, or {\it{KAN-algebra}}, is an algebra $\langle T, \vee, \wedge, \sim, \neg, 0, 1 \rangle$ of type $(2,2,1,1,0,0)$ such that $\langle T, \vee, \wedge, \sim, 0, 1 \rangle$ is a Kleene algebra and the following conditions are satisfied:
\begin{itemize}
\item [(N1)] $\neg (x \wedge \neg (x \wedge y)) = \neg (x \wedge \neg y)$,
\item [(N2)] $\neg (x \vee y) = \neg x \wedge \neg y$,
\item [(N3)] $x \wedge \sim x = x \wedge \neg x$,
\item [(N4)] $\sim x \leq \neg x$,
\item [(N5)] $\neg (x \wedge y) = \neg ((\sim \neg x) \wedge y)$.
\end{itemize}
An algebra $\langle T, \vee, \wedge, \sim, \neg, c, 0, 1 \rangle$ of type $(2,2,1,1,0,0,0)$ is a {\it{centered KAN-algebra}}, or {\it{KANc-algebra}}, if $\langle T, \vee, \wedge, \sim, \neg, 0, 1 \rangle$ is a KAN-algebra and the structure $\langle T, \vee, \wedge, \sim, c, 0, 1 \rangle$ is a centered Kleene algebra.
\end{definition}

We denote the variety of KAN-algebras by $\bf{KAN}$ and the variety of centered KAN-algebras by $\bf{KANc}$.

\begin{example} 
Let $\langle T, \vee, \wedge, \sim, \to, 0, 1 \rangle$ be a Nelson algebra, and define $\neg x := x \to 0$. Then, the structure $\langle T, \vee, \wedge, \sim, \neg, 0, 1 \rangle$ is a KAN-algebra.
\end{example}

\begin{example} \label{noStoneKAN} 
Consider the bounded distributive lattice $T$ given by the following Hasse diagram:
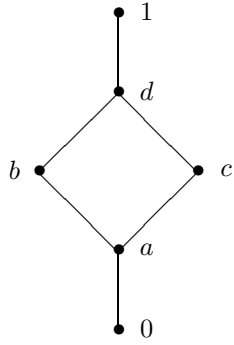
\vspace{0.25cm}
\begin{center}
\hspace{0.25cm}
\put(00,00){\makebox(1,1){$\bullet$}}
\put(00,30){\makebox(1,1){$\bullet$}}
\put(-30,60){\makebox(1,1){$\bullet$}}
\put(30,60){\makebox(1,1){$\bullet$}}
\put(00,90){\makebox(1,1){$\bullet$}}
\put(00,120){\makebox(1,1){$\bullet$}} 
\put(00,00){\line(0,1){30}}
\put(00,30){\line(-1,1){30}}
\put(00,30){\line(1,1){30}}
\put(-30,60){\line(1,1){30}}
\put(30,60){\line(-1,1){30}}
\put(00,90){\line(0,1){30}} 
\put(10,00){\makebox(2,2){$0$}}
\put(10,30){\makebox(2,2){$a$}}
\put(-40,60){\makebox(2,2){$b$}} 
\put(40,60){\makebox(2,2){$c$}} 
\put(10,90){\makebox(2,2){$d$}}
\put(10,120){\makebox(2,2){$1$}} 
\end{center}
\vspace{0.25cm}
If we define the operators $\sim$ and $\neg$ as follows
\vspace{0.25cm}
\begin{center}
\begin{tabular}{|c|c|c|c|c|c|c|c|}\hline 
$x$          & $0$ & $a$ & $b$ & $c$ & $d$ & $1$ \\ \hline
$\sim x$  & $1$ & $d$  & $c$ & $b$ & $a$ & $0$ \\ \hline
$\neg x$  & $1$ & $1$ & $c$ & $b$ & $a$ & $0$ \\ \hline
\end{tabular}
\end{center}
\vspace{0.25cm}
then it can be verified that $\langle T,\vee,\wedge,\neg,\sim,0,1 \rangle$ is a KAN-algebra.
\end{example}

\begin{remark}
If $T \in {\bf{KAN}}$ is finite, then it is isomorphic to the reduct of a Nelson algebra. However, not every KAN-algebra is a reduct of a Nelson algebra (for details, see \cite{Conrado-Miguel-Hernan}). 
\end{remark}

A Kleene algebra $\langle T, \vee, \wedge, \sim, 0, 1 \rangle$ is said to be quasi weak pseudocomplemented if for each $x \in T$, there exists a maximum element in the set 
\[
\{ y \in T \colon x \wedge y \leq \sim x \},
\]
which is denoted by $\neg x$. The element $\neg x$ is characterized by the following property:  
\[
y \leq \neg x \Longleftrightarrow x \wedge y \leq \sim x,
\]
for all $y \in T$. The unary operation $\neg$ defined in this way is called a quasi weak pseudocomplementation and the algebra $\langle T, \vee, \wedge, \sim, \neg, 0, 1 \rangle$, of type $(2, 2, 1, 1, 0, 0)$, is called a {\it{Kleene algebra with a quasi weak pseudocomplementation}}, or simply a {\it{qwp-Kleene algebra}}. A qwp-Kleene algebra is called a \emph{Kleene algebra with a weak pseudocomplementation}, or \emph{wp-Kleene algebra}, if for each $x, y \in T$ we have 
\[
\neg(x \wedge y) = 1 \Longleftrightarrow \neg\neg x \leq \neg y.
\]
The class of wp-Kleene algebras is denoted by $\mathcal{K}_{\omega}$ (\cite{Sendlewski1}). Furthermore, the variety of KAN-algebras coincides with the class $\mathcal{K}_{\omega}$ as shown in \cite{Conrado-Miguel-Hernan}.

For any $T \in {\bf{KAN}}$ we adopt the following notation:
\begin{equation} \label{Delta}
\lozenge x := \sim \neg x 
\end{equation}
\begin{equation} \label{nabla}
\square x := \neg \sim x   
\end{equation}
It is easy to verify that $\square x = \sim \lozenge \sim x$ and $\lozenge x = \sim \square \sim x$. Using these operators, we can describe several important properties of KAN-algebras. Next, we will present the results from \cite{Conrado-Miguel-Hernan} that will be used in this paper and prove new results that are fundamental for the following sections.

\begin{proposition} \cite{Conrado-Miguel-Hernan} \label{Prop1}
Let $T \in {\bf{KAN}}$. Then:
\begin{enumerate}
\item $\neg 1=0$ and $\neg 0=1$,
\item $x\leq y$ implies $\neg y \leq \neg x$,
\item $\neg (x \wedge \sim x) = \neg (x \wedge \neg x) = 1$,
\item $\lozenge 0=0$ and $\square 1=1$,
\item $\lozenge 1=1$ and $\square 0=0$,
\item $\lozenge x \leq x \leq \square x$,
\item $\lozenge x \leq \neg \neg x \leq \square x$,
\item $\neg \lozenge x = \square \neg x = \neg x$.
\end{enumerate}
\end{proposition}

\begin{lemma} \cite{Conrado-Miguel-Hernan} \label{Prop2}
Let $T \in {\bf{KAN}}$. Then $x \leq y$ if and only if $\lozenge x \leq \lozenge y$ and $\square x \leq \square y$.
\end{lemma}

From Lemma \ref{Prop2} we have the following result.

\begin{lemma} (Moisil's determination principle) \label{MDP}
Let $T \in {\bf{KAN}}$. If $\square x = \square y$ and $\lozenge x = \lozenge y$, then $x = y$.    
\end{lemma}

\begin{lemma} \cite{Conrado-Miguel-Hernan}\label{lemma5.8}
Let $T \in {\bf{KAN}}$. Then $\neg(x \wedge y) = 1$ if and only if $\neg\neg x \leq \neg y$.
\end{lemma}

\begin{lemma}  \label{auxiliar1}
If $T \in {\bf{KAN}}$. Then $\lozenge(x \wedge y) = 0$ if and only if $\lozenge x \leq \lozenge \neg y$.
\end{lemma}
\begin{proof} 
By Lemma \ref{lemma5.8} we have that
\[
\begin{array}{lllllll}
\lozenge(x \wedge y) = 0  & \Longleftrightarrow & \neg(y \wedge x) = 1                   & \Longleftrightarrow & \neg \neg y \leq \neg x  \\
                                     & \Longleftrightarrow & \sim \neg x \leq \sim \neg \neg y  & \Longleftrightarrow & \lozenge x \leq \lozenge \neg y, \\
\end{array}
\]
i.e., $\lozenge(x \wedge y) = 0$ if and only if $\lozenge x \leq \lozenge \neg y$.
\end{proof}

\begin{lemma} \label{t}
Let $ T \in \mathbf{KAN} $. Then $\lozenge(x \wedge \sim x) = 0 $.
\end{lemma}
\begin{proof}
This follows from Lemma \ref{auxiliar1} and condition (N3).
\end{proof}

\begin{lemma} \label{lemma2.23}
Let $T \in {\bf{KAN}}$. If $\lozenge x = x$, $\lozenge y = y$ and $\lozenge (x \wedge y) = 0$, then $x \leq \sim y$.
\end{lemma}
\begin{proof} 
Let $x, y \in T$ be such that $\lozenge x = x$, $\lozenge y = y$ and $\lozenge (x \wedge y) = 0$. It follows $\neg x = \sim x$, $\neg y = \sim y$ and $\neg (x \wedge y) = 1$. From this last statement and Lemma \ref{lemma5.8} we obtain $\neg \neg x \leq \neg y$. So, $\neg \sim x \leq \sim y$ and by item (6) of Proposition \ref{Prop1}, we conclude that $x \leq \sim y$, as required.  
\end{proof}

The next results will be useful to us.

\begin{proposition} \label{Prop_3}
Let $T \in {\bf{KAN}}$. Then:
\begin{enumerate}
\item $\lozenge \lozenge x = \lozenge x$,
\item $\lozenge (x \vee y) = \lozenge x \vee \lozenge y$,
\item $\lozenge (\lozenge x \wedge \lozenge y) = \lozenge (x \wedge y)$,
\item $\square \square x = \square x$,
\item $\square (x \wedge y) = \square x \wedge \square y$,
\item $\square (\square x \vee \square y) = \square (x \vee y)$,
\item $x \vee \sim x = \sim x \vee \lozenge x$,
\item $x \wedge \sim x = x \wedge \square \sim x$,
\item $x = (\lozenge x \vee \sim x) \wedge \square x$.
\end{enumerate}
\end{proposition}
\begin{proof}
(1): By (8) of Proposition \ref{Prop1} we have $\lozenge \lozenge x = \sim \neg \lozenge x = \sim \neg x = \lozenge x$.

(2): By (N2) we have 
\[
\begin{array}{lllllll}
\lozenge(x \vee y)  & = & \sim \neg (x \vee y) & = & \sim (\neg x \wedge \neg y) \\
                          & = & \sim \neg x \vee \sim \neg y & = & \lozenge x \vee \lozenge y. \\
\end{array}
\]

(3): By (N5),
\[
\begin{array}{lllllll}
\lozenge (\lozenge x \wedge \lozenge y)  & = & \sim \neg ( \sim \neg x \wedge \sim \neg y) & = & \sim \neg (x \wedge \sim \neg y) \\
                                                     & = & \sim \neg (x \wedge y) & = & \lozenge (x \wedge y). \\
\end{array}
\]

(4): By (8) of Proposition \ref{Prop1} it follows $\square \square x = \square \neg \sim x = \neg \sim x = \square x$.

(5): By (N2) we have 
\[
\begin{array}{lllllll}
\square (x \wedge y)  & = & \neg \sim (x \wedge y)                  & = & \neg ( \sim x \vee \sim y) \\
                                & = & \neg \sim x \wedge \neg \sim y     & = & \square x \wedge \square y. \\
\end{array}
\]

(6): Then, by (N5)  
\[
\begin{array}{lllllll}
\square (\square x \vee \square y)  & = & \neg \sim (\neg \sim x \vee \neg \sim y) \\
                                                 & = & \neg (\sim \neg \sim x \wedge \sim \neg \sim y) \\
                                                 & = & \neg (\sim x \wedge \sim \neg \sim y) \\
                                                 & = & \neg (\sim x \wedge \sim y) \\
                                                 & = & \neg \sim (x \vee y) \\
                                                 & = & \square (x \vee y). \\
\end{array}
\]

(7): By (N3),  
\[
\begin{array}{lllllll}
x \vee \sim x  & = & \sim \sim x \vee \sim x    & = & \sim (x \wedge \sim x) \\
                      & = & \sim (x \wedge \neg x)    & = & \sim x \vee \lozenge x. \\
\end{array}
\]

(8): It follows by (N3) and Definition \ref{nabla}.

(9): By items (7), (8) and Proposition \ref{Prop1} we have 
\[
\begin{array}{lllllll}
(\lozenge x \vee \sim x) \wedge \square x  & = & (x \vee \sim x) \wedge \square x   \\
                                                             & = & (x \wedge \square x) \vee (\sim x \wedge \square x) \\
                                                             & = & x \vee (\sim x \wedge \square x) \\
                                                             & = & x \vee (\sim x \wedge x) \\
                                                             & = & x. \\
\end{array}
\]
This concludes the proof.
\end{proof}

The following result lists some properties that hold in centered KAN-algebras.

\begin{lemma} \label{auxiliar2} 
Let $T \in {\bf{KANc}}$. Then:
\begin{enumerate}
\item $\neg c=1$,
\item $\lozenge c = 0$ and $\square c = 1$,
\item $\lozenge(c \wedge \square x) = 0$ and $\square(\lozenge x \vee c) = 1$,
\item $x = (\lozenge x \vee c) \wedge \square x$,
\item $x \vee c = \lozenge x \vee c$.
\end{enumerate}
\end{lemma}
\begin{proof}
(1): The proof can be found in \cite{Conrado-Miguel-Hernan}. 

(2): By item (1) and the definitions \ref{Delta} and \ref{nabla} we have 
\[
\lozenge c = \sim \neg c = \sim 1 = 0 \text{ and } \square c = \neg \sim c = \neg c = 1.
\]

(3): Since $c \wedge \lozenge x \leq c$ and $c \leq \lozenge x \vee c$, it follows from Lemma \ref{Prop2} that $\lozenge (c \wedge \square x) \leq \lozenge c$ and $\square c \leq \square (\lozenge x \vee c)$. Then, by item (2), we have $\lozenge (c \wedge \square x) = 0$ and $\square (\lozenge x \vee c) = 1$, which proves (3). 

(4): We will use the Moisil's determination principle. We show $\lozenge x = \lozenge ((\lozenge x \vee c) \wedge \square x)$ and $\square x = \square ((\lozenge x \vee c) \wedge \lozenge x)$. By Propositions \ref{Prop1}, \ref{Prop_3} and item (3) we have
\[
\begin{array}{lllllll}
\lozenge ((\lozenge x \vee c) \wedge \square x)  & = & \lozenge ((\lozenge x \wedge \square x) \vee (c \wedge \square x)) \\
                                                                  & = & \lozenge (\lozenge x \vee (c \wedge \square x)) \\
                                                                  & = & \lozenge \lozenge x \vee \lozenge (c \wedge \square x) \\
                                                                  & = & \lozenge x \\
\end{array}
\]
and $\square ((\lozenge x \vee c) \wedge \square x)  =  \square (\lozenge x \vee c) \wedge \square \square x  =  \square x$. Therefore, $x = (\lozenge x \vee c) \wedge \square x$.

(5): By Proposition \ref{Prop1} we have $\lozenge x \leq \square x$. Thus, by item (4), 
\[
\begin{array}{lllllll}
x \vee c  & = & ((\lozenge x \vee c) \wedge \square x) \vee c & = & (\lozenge x \vee c) \wedge (\square x \vee c) \\
              & = & (\lozenge x \wedge \square x) \vee c              & = & \lozenge x \vee c, \\
\end{array}
\]
i.e., $x \vee c = \lozenge x \vee c$ as desired. 
\end{proof}

\begin{lemma} \label{lemacentro}  
Let $T \in {\bf{KAN}}$ and $c \in T$. Then $c$ is the center of $T$ if and only if $\lozenge c = 0$ and $\square c = 1$.
\end{lemma}
\begin{proof}
If $c$ is the center of $T$, by Proposition \ref{auxiliar2} we have $\lozenge c=0$ and $\square c=1$. Conversely, if we suppose $\lozenge c=0$ and $\square c=1$, then by Proposition \ref{Prop_3} we have $c=(\lozenge c \vee \sim c) \wedge \square c = (0 \vee \sim c) \wedge 1 = \sim c$ and $c$ is the center of $T$.
\end{proof}

\begin{remark}
A \emph{modal De Morgan algebra} is a pair $\langle A,\square \rangle$, where $A$ is a De Morgan algebra and $\square$ is a unary operator satisfying 
\[
\square 1 = 1 
\quad\text{and}\quad 
\square(x \wedge y) = \square x \wedge \square y.
\]
By defining $\lozenge x := \sim\square\sim x$, one may equivalently describe a modal De Morgan algebra as a pair $\langle A, \lozenge \rangle$ satisfying 
\[
\lozenge 0 = 0
\quad\text{and}\quad
\lozenge(x \vee y) = \lozenge x \vee \lozenge y.
\]
These operators are typically referred to as \emph{normal modal operators} (see \cite{Celani2011,Gregori2019}). Moreover, based on the results of Propositions~\ref{Prop1} and \ref{Prop_3}, we can conclude that $\lozenge$ and $\square$ are normal modal operators.
\end{remark}

\section{Stone KAN-algebras} \label{sec3}

In this section, we define the variety of Stone-Kleene algebras with intuitionistic negation and explore their relationship with the variety of Stone algebras.

\begin{definition} 
Let $T \in {\bf{KAN}}$. We say that $T$ is a {\it{Stone-Kleene algebra with intuitionistic negation}}, or a {\it{Stone KAN-algebra}}, if it satisfies the equation
\begin{equation} \label{stoneKAN}
\neg x \vee \neg \neg x = 1  
\end{equation}
for all $x \in A$. A Stone KAN-algebra $T$ is called a {\it{centered Stone KAN-algebra}}, or {\it{Stone KANc-algebra}}, if $T$ is a centered KAN-algebra with center $c$.
\end{definition}

Denote the variety of Stone KAN-algebras by $\bf{SKAN}$ and the variety of centered Stone KAN-algebras by $\bf{SKANc}$.

\begin{example}
Let $T$ be the KAN-algebra from Example \ref{noStoneKAN}. Since $\neg\neg b \vee \neg b \neq 1$, $T$ is not a Stone KAN-algebra. Thus, the variety ${\bf{SKAN}}$ is a proper subvariety of ${\bf{KAN}}$.
\end{example}

\begin{example} \label{cadena3}
Consider the three-element chain $T = \{0, c, 1\}$ represented by the following Hasse diagram: 
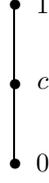
\vspace{0.25cm}
\begin{center}
\hspace{0.25cm}
\put(00,10){\makebox(1,1){$\bullet$}}
\put(00,40){\makebox(1,1){$\bullet$}}
\put(00,70){\makebox(1,1){$\bullet$}}
\put(00,10){\line(0,0){30}}
\put(00,40){\line(0,0){30}}
\put(10,10){\makebox(2,2){$ 0$}} 
\put(10,40){\makebox(2,2){$c$}}
\put(10,70){\makebox(2,2){$1$}}
\end{center}
Define the unary operators $\sim$ and $\neg$ as follows
\vspace{0.25cm}
\begin{center}
\begin{tabular}{|c|c|c|c|}\hline
$x$           & $0$ & $c$ & $1$ \\ \hline
$\sim x$   & $1$ & $c$ & $0$ \\ \hline
$\neg x$   & $1$ & $1$ & $0$ \\ \hline
\end{tabular}
\end{center}
\vspace{0.25cm}
It is easy so see that $\langle T, \vee, \wedge, \sim, \neg, c,0, 1 \rangle$ is a centered Stone KAN-algebra.
\end{example}

The following example shows a Stone KAN-algebra that is not centered.

\begin{example} \label{SKANalgebra} 
Consider the bounded distributive lattice $T$ given by the following Hasse diagram:
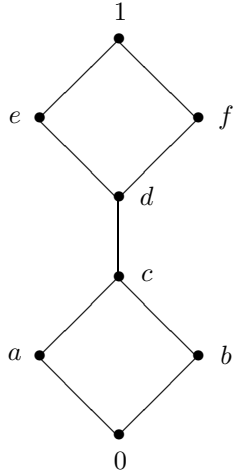
\vspace{0.25cm}
\begin{center}
\hspace{0.25cm}
\put(30,40){\makebox(1,1){$\bullet$}}
\put(-30,40){\makebox(1,1){$\bullet$}}
\put(00,10){\makebox(1,1){$\bullet$}}
\put(00,70){\makebox(1,1){$\bullet$}}
\put(00,100){\makebox(1,1){$\bullet$}}
\put(-30,130){\makebox(1,1){$\bullet$}}
\put(30,130){\makebox(1,1){$\bullet$}}
\put(00,160){\makebox(1,1){$\bullet$}}
\put(-30,40){\line(1,1){30}}
\put(30,40){\line(-1,1){30}}
\put(00,10){\line(1,1){30}}
\put(00,10){\line(-1,1){30}}
\put(00,70){\line(0,0){30}}
\put(00,100){\line(1,1){30}}
\put(00,100){\line(-1,1){30}}
\put(-30,130){\line(1,1){30}}
\put(30,130){\line(-1,1){30}}
\put(00,00){\makebox(2,2){$0$}}
\put(-40,40){\makebox(2,2){$a$}} 
\put(40,40){\makebox(2,2){$b$}} 
\put(10,70){\makebox(2,2){$c$}}
\put(10,100){\makebox(2,2){$d$}}
\put(-40,130){\makebox(2,2){$e$}} 
\put(40,130){\makebox(2,2){$f$}} 
\put(00,170){\makebox(2,2){$1$}}
\end{center}
\vspace{0.25cm}
We define the operators $\sim$ and $\neg$ as follows
\vspace{0.25cm}
\begin{center}
\begin{tabular}{|c|c|c|c|c|c|c|c|c|c|}\hline 
$x$          & $0$ & $a$ & $b$ & $c$ & $d$ & $e$ & $f$ & $1$ \\ \hline
$\sim x$  & $1$ & $f$  & $e$ & $d$ & $c$ & $b$ & $a$ & $0$ \\ \hline
$\neg x$  & $1$ & $1$ & $1$ & $1$ & $c$ & $b$ & $a$ & $0$ \\ \hline
\end{tabular}
\end{center}
\vspace{0.25cm}
Then it can be verified that $\langle T,\vee,\wedge,\neg,\sim,0,1 \rangle$ is a Stone KAN-algebra.
\end{example}

Let $T \in {\bf{KAN}}$. Consider the set $T^{\lozenge} := \{ x \in T \colon \lozenge x = x \}$ and define the following operations on $T^{\lozenge}$:
\begin{itemize}
\item $x \vee^{\lozenge} y := \lozenge (x \vee y)$,
\item $x \wedge^{\lozenge} y := \lozenge (x \wedge y)$,
\item $\neg^{\lozenge} x := \lozenge \neg x$.
\end{itemize}
By Proposition \ref{Prop_3} it follows that $x \vee^{\lozenge} y = x \vee y$. Note that if $x \in T^{\lozenge}$, then $\sim x = \neg x$ and hence $T^{\lozenge} = \{ x \in T \colon \sim x = \neg x \}$. Furthermore, if we consider the set $T^{\square} := \{ x \in T \colon \square x = x \}$, it is easy to see that $T^{\square} = \sim (T^{\lozenge})$.

\begin{theorem} \label{teorema2.11.}
Let $T \in {\bf{SKAN}}$. Then $\langle T^{\lozenge}, \vee, \wedge^{\lozenge}, \neg^{\lozenge}, 0, 1 \rangle$ is a Stone algebra.
\end{theorem}
\begin{proof}
By Lemma \ref{Prop_3} we can show that $ \langle T^{\lozenge}, \vee, \wedge^{\lozenge}, 0, 1 \rangle $ is a bounded distributive lattice. Furthermore, if $x,y \in T^{\lozenge}$, then $\lozenge x = x$ and by Lemma \ref{auxiliar1} we know that $x \wedge^{\lozenge} y = 0$ if and only if $x \leq \neg^{\lozenge} y$. Thus, $T^{\lozenge}$ is a distributive p-algebra. It remains to prove that $\neg^{\lozenge} x \vee \neg^{\lozenge} \neg^{\lozenge} x = 1$, for all $x \in T^{\lozenge}$. Since equation \ref{stoneKAN} holds, and by Propositions \ref{Prop1} and \ref{Prop_3}, we have
\[
\begin{array}{lllllll}
\neg^{\lozenge} x \vee \neg^{\lozenge} \neg^{\lozenge} x  & = & \lozenge \neg x \vee \lozenge \neg \lozenge \neg x \\
                                                                               & = & \lozenge \neg x \vee \lozenge \neg \neg x \\
                                                                              & = & \lozenge(\neg x \vee \neg \neg x)  \\
                                                                             & = & 1. \\
\end{array}
\]
Therefore, $T^{\lozenge}$ is a Stone algebra.
\end{proof}

\begin{proposition}
Let $T_1, T_2 \in {\bf SKAN}$. If $f \colon T_{1} \to T_{2}$ is a morphism in {\bf{SKAN}}, then $f^{\lozenge} \colon T_{1}^{\lozenge} \to T_{2}^{\lozenge}$ defined by $f^{\lozenge}(x) = \lozenge f(x)$ is a morphism in {\bf{Stone}}.
\end{proposition}
\begin{proof}
We prove that $f^{\lozenge}$ is compatible with the operation $\neg^{\lozenge}$. If $x \in T_{1}^{\lozenge}$, then by Proposition \ref{Prop1} we have
\[
\neg^{\lozenge} f^{\lozenge}(x) = \lozenge \neg \lozenge f(x) = \lozenge \neg f(x) = f(\neg^{\lozenge} x).
\]
The remainder of the proof follows directly from Proposition \ref{Prop_3}.
\end{proof}

\begin{example} \label{ejdelta} 
Consider the Stone KAN-algebra from Example \ref{SKANalgebra}. So, we have the Stone algebra $T^{\lozenge} = \{0, d, e, f, 1\}$, given by
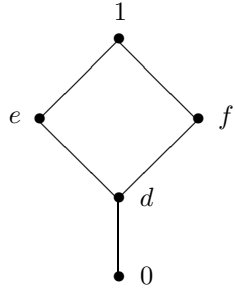
\vspace{0.25cm}
\begin{center}
\hspace{0.25cm}
\put(00,00){\makebox(1,1){$\bullet$}}
\put(00,30){\makebox(1,1){$\bullet$}}
\put(-30,60){\makebox(1,1){$\bullet$}}
\put(30,60){\makebox(1,1){$\bullet$}}
\put(00,90){\makebox(1,1){$\bullet$}}
\put(00,00){\line(0,0){30}}
\put(00,30){\line(1,1){30}}
\put(00,30){\line(-1,1){30}}
\put(-30,60){\line(1,1){30}}
\put(30,60){\line(-1,1){30}}
\put(10,00){\makebox(2,2){$0$}}
\put(10,30){\makebox(2,2){$d$}}
\put(-40,60){\makebox(2,2){$e$}} 
\put(40,60){\makebox(2,2){$f$}} 
\put(00,100){\makebox(2,2){$1$}}
\end{center}
where
\vspace{0.25cm}
\begin{center}
\begin{tabular}{|c|c|c|c|c|c|c|c|c|c|}\hline 
$x$                      & $0$ & $d$ & $e$ & $f$ & $1$  \\ \hline
$\neg^{\lozenge}x$  & $1$ & $0$  & $0$ & $0$ & $0$  \\ \hline
\end{tabular}
\end{center}
\vspace{0.1cm}
\end{example}

There exist a construction that relate distributive p-algebras and KAN-algebras. In \cite{Conrado-Miguel-Hernan}, the authors proved that if $T$ is a KAN-algebra, then the binary relation $\theta \subseteq T \times T$ defined by 
\[
(x,y) \in \theta \Longleftrightarrow \neg x = \neg y
\]
is an equivalence relation compatible with the operations $\vee$, $\wedge$ and $\neg$. We denote by $[x]_{\theta}$ the equivalence class of $x$ modulo $\theta$ and $T^{\theta} = \{ [x]_{\theta} \colon x \in T \}$ the set of all equivalence classes. So, the structure $\langle T^{\theta}, \vee^{\theta}, \wedge^{\theta}, \neg^\theta, [0]_\theta, [1]_\theta \rangle$ is a distributive p-algebra, where the operations on $T^{\theta}$ are defined by: 
\begin{itemize}
\item $[x]_{\theta} \vee^{\theta} [y]_{\theta} := [x \vee y]_{\theta}$,
\item $[x]_{\theta} \wedge^{\theta} [y]_{\theta} := [x \wedge y]_{\theta}$,
\item $\neg^{\theta} [x]_{\theta} := [ \neg x]_{\theta}$.
\end{itemize}
The order $\leq$ in $T^{\theta}$ can be characterized as $[x]_{\theta} \leq [y]_{\theta}$ if and only if $\neg y \leq \neg x$. Moreover, if $T$ is a Stone KAN-algebra, we have 
\[
\neg^{\theta} [x]_{\theta} \vee^{\theta} \neg^{\theta} \neg^{\theta} [x]_{\theta} = [\neg x \vee \neg \neg x]_{\theta} = [1]_{\theta}
\]
and $T^{\theta}$ satisfies the equation \ref{stone}. Thus, the following result holds:

\begin{theorem} 
Let $T \in {\bf{SKAN}}$. Then $\langle T^{\theta}, \vee^{\theta}, \wedge^{\theta}, \neg^{\theta}, [0]_{\theta}, [1]_{\theta} \rangle$ is a Stone algebra.
\end{theorem}

\begin{remark}
By previous results, we know that if $T \in {\bf{SKAN}}$, then $T^{\theta}$ is a Stone algebra. Moreover, if $f \colon T_1 \to T_2$ is a morphism in ${\bf{SKAN}}$, then the map $f^{\theta} \colon T_1^{\theta} \to T_2^{\theta}$ defined by $f^{\theta}([x]_\theta) = [f(x)]_\theta$ for all $x \in T_1$ is a morphism between Stone algebras. Therefore, $\theta$ is a functor from ${\bf{SKAN}}$ to the class of Stone algebra.
\end{remark}

The following theorem illustrates the relationship between the two previously described constructions.

\begin{theorem} \label{t2.13}
Let $T \in {\bf{SKAN}}$. Then the Stone algebras $\langle T^{\lozenge}, \vee, \wedge^{\lozenge}, \neg^{\lozenge}, 0, 1 \rangle$ and $\langle T^{\theta}, \vee^{\theta}, \wedge^{\theta}, \neg^{\theta}, [0]_{\theta}, [1]_{\theta} \rangle$ are isomorphic.
\end{theorem}
\begin{proof}
Consider the map $\varphi \colon T^{\theta} \to T^{\lozenge}$ given by $\varphi([x]_{\theta}) := \lozenge x$. Then, by Proposition \ref{Prop_3}, we have 
\[
\varphi([x]_{\theta} \vee^{\theta} [y]_{\theta}) = \lozenge(x \vee y) = \lozenge x \vee \lozenge y = \varphi([x]_{\theta}) \vee^{\theta} \varphi([y]_{\theta})
\]
and
\[
\begin{array}{llllll}
\varphi([x]_{\theta} \wedge^{\theta} [y]_{\theta}) & = & \lozenge(x \wedge y) \\ 
                                                                           & = & \lozenge (\lozenge x \wedge \lozenge y) \\ 
                                                                           & = & \lozenge(\varphi([x]_{\theta}) \wedge \varphi([y]_{\theta})) \\
                                                                           & = & \varphi([x]_{\theta}) \wedge^{\lozenge} \varphi([y]_{\theta}).
\end{array}
\]
On the other hand, by Proposition \ref{Prop1}, 
\[
\varphi(\neg^{\theta} [x]_{\theta}) = \varphi([\neg x]_{\theta}) = \lozenge \neg x = \lozenge \neg \lozenge x = \neg^{\lozenge} \lozenge x = \neg^{\lozenge} \varphi([x]_{\theta}), 
\]
i.e., $\varphi(\neg^{\theta} [x]_{\theta}) = \neg^{\lozenge} \varphi([x]_{\theta})$. This completes the proof.
\end{proof}

\subsection{Kalman's construction for Stone algebras}

Now, we present a Kalman's construction for Stone algebras. Using this construction, we prove that the category of Stone algebras is equivalent to the category of Stone KANc-algebras.

Let $\langle A,\vee,\wedge,^{\ast},0,1\rangle$ be a distributive p-algebra and let us consider the set 
\[
K(A) := \{ (x,y) \in A \times A \colon x \wedge y = 0 \}
\]
endowed with the following operations:
\begin{eqnarray*}
(a,b) \vee (d,e)        & := & (a \vee d, b \wedge e), \\
(a,b) \wedge (d,e)   & := & (a \wedge d, b \vee e), \\
\sim (a,b)                & := & (b, a), \\
\neg (a,b)                & := & (a^{\ast}, a), \\
0_K                             & := & (0,1), \\
1_K                             & := & (1,0), \\ 
c_K                             & := & (0,0). \\ 
\end{eqnarray*}
From the results developed in \cite{Conrado-Miguel-Hernan} we conclude that the structure 
\[
\langle K(A), \vee, \wedge, \sim, \neg, c_K, 0_K, 1_K \rangle
\]
is a centered KAN-algebra, where the order on $K(A)$ is given by $(a, b) \leq (d, e)$ if and only if $a \leq d$ and $e \leq b$. The following lemma provides  necessary and sufficient conditions for $K(A)$ to be a Stone KANc-algebra.

\begin{lemma}
Let $A$ be a distributive p-algebra. Then $A$ is a Stone algebra if and only if $K(A)$ is a Stone KANc-algebra.
\end{lemma}
\begin{proof}
Suppose that $A$ is a Stone algebra, i.e., satisfies the equation \ref{stone}. Let $(a,b) \in K(A)$. Then 
\[
\begin{array}{lllllll}
\neg (a,b) \vee \neg \neg (a,b) & = & (a^{\ast}, a) \vee (a^{\ast\ast}, a^{\ast}) \\
                                                & = & (a^{\ast} \vee a^{\ast\ast}, a \wedge a^{\ast}) \\
                                                & = & (1,0)  
\end{array}
\]
and $K(A)$ is a Stone KAN-algebra. Conversely, assume that $K(A)$ is a Stone KANc-algebra. Then $K(A)$ satisfies the equation \ref{stoneKAN}. If $x \in A$, then $(x, 0) \in K(A)$ and 
\[
(1,0) = \neg (x,0) \vee \neg \neg (x,0) = (x^{\ast} \vee x^{\ast\ast}, x \wedge x^{\ast})
\]
which implies $x^{\ast} \vee x^{\ast\ast} = 1$. So, $A$ is a Stone algebra.
\end{proof}

\begin{example} \label{ejKalman} 
Consider the Stone KAN-algebra $\langle T, \vee, \wedge, \sim, \neg, 0, 1 \rangle$ from Example \ref{SKANalgebra} and the Stone algebra $T^{\lozenge}$ associated with $\langle T, \vee, \wedge, \sim, \neg, 0, 1 \rangle$ from Example \ref{ejdelta}. Then, by Kalman's construction, we obtain the centered Stone KAN-algebra $K(T^{\lozenge})$ given by
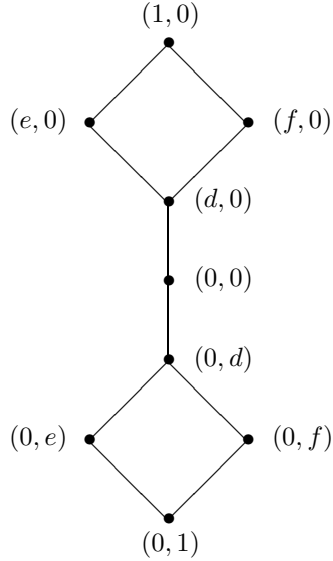
\vspace{0.5cm}
\begin{center}
\hspace{0.25cm}
\put(00,00){\makebox(1,1){$\bullet$}}
\put(-30,30){\makebox(1,1){$\bullet$}}
\put(30,30){\makebox(1,1){$\bullet$}}
\put(00,60){\makebox(1,1){$\bullet$}}
\put(00,90){\makebox(1,1){$\bullet$}}
\put(00,120){\makebox(1,1){$\bullet$}}
\put(-30,150){\makebox(1,1){$\bullet$}}
\put(30,150){\makebox(1,1){$\bullet$}}
\put(00,180){\makebox(1,1){$\bullet$}}
\put(00,00){\line(1,1){30}}
\put(00,00){\line(-1,1){30}}
\put(-30,30){\line(1,1){30}}
\put(30,30){\line(-1,1){30}}
\put(00,60){\line(0,0){30}}
\put(00,90){\line(0,0){30}}
\put(00,120){\line(1,1){30}}
\put(00,120){\line(-1,1){30}}
\put(-30,150){\line(1,1){30}}
\put(30,150){\line(-1,1){30}}
\put(00,-10){\makebox(2,2){$ (0,1)$}}
\put(-50,30){\makebox(2,2){$ (0,e)$}} 
\put(50,30){\makebox(2,2){$ (0,f)$}} 
\put(20,60){\makebox(2,2){$ (0,d)$}} 
\put(20,90){\makebox(2,2){$(0,0)$}}
\put(20,120){\makebox(2,2){$ (d,0)$}}
\put(-50,150){\makebox(2,2){$ (e,0)$}} 
\put(50,150){\makebox(2,2){$ (f,0)$}} 
\put(00,190){\makebox(2,2){$ (1,0)$}}
\end{center}
\vspace{0.2cm}
\end{example}

Let $T \in {\bf{KANc}}$. Consider the set $C(T) := \{ x \in T \colon x \geq c \}$ and define the unary operation $\neg^{c} x := \neg x \vee c$ on $C(T)$.

\begin{lemma} \label{lemaant}
Let $T \in {\bf{SKANc}}$. Then $\langle C(T), \vee, \wedge, \neg^{c}, c, 1 \rangle$ is a Stone algebra. 
\end{lemma}
\begin{proof} 
By the results given in \cite{Conrado-Miguel-Hernan} we have that $\langle C(T), \vee, \wedge, \neg^{c}, c, 1 \rangle$ is a distributive p-algebra. We prove that $\neg^{c}x \vee \neg^{c}\neg^{c}x = 1$, for all $x \in C(T)$. If $x \in C(T)$, then since the equation \ref{stoneKAN} holds, then by (N2) and Lemma \ref{auxiliar2} we have 
\[
\begin{array}{lllllll}
\neg^{c}x \vee \neg^{c}\neg^{c}x & = & (\neg x \vee c) \vee (\neg(\neg x \vee c)) \vee c \\ 
                                                    &=& (\neg x \vee c) \vee (\neg \neg x \wedge 1) \vee c \\
                                                    & = & \neg x \vee \neg \neg x \vee c  \\
                                                    &=& 1. \\
\end{array}
\]
Therefore, $C(T)$ is a Stone algebra.
\end{proof}

From Theorem \ref{t2.13} and \cite{Conrado-Miguel-Hernan} we obtain the following results.

\begin{theorem} \label{C(T)=H} 
Let $T \in {\bf{SKANc}}$. Then the Stone algebras $\langle C(T), \vee, \wedge, \neg^{c}, c, 1 \rangle$ and $\langle T^{\lozenge}, \vee, \wedge^{\lozenge}, \neg^{\lozenge}, 0, 1 \rangle$ are isomorphic.
\end{theorem}
\begin{proof} 
Let $h \colon T^{\lozenge} \to C(T)$ be the map given by $h(x) := x \vee c$. So, $h(0) = c$, $h(1) = 1$ and $h(x \vee y) = h(x) \vee h(y)$ are straightforward. By Lemma \ref{auxiliar2} we have
\[
\begin{array}{lllllll}
h(x \wedge^{\lozenge} y) & = & (x \wedge^{\lozenge} y) \vee c &=& \lozenge(x \wedge y) \vee c \\
                                    & = & (x \wedge y) \vee c               &=& (x \vee c) \wedge (y \vee c) \\
                                    & = & h(x) \wedge h(y).
\end{array}
\]
On the other hand, by Lemma \ref{auxiliar2}, (N2) and Proposition \ref{Prop1}, it follows  
\[
\begin{array}{lllllll}
\neg^{c} h(x) & = & \neg (x \vee c) \vee c       &=& (\neg x \wedge \neg c) \vee c \\
                     & = & \neg x \vee c                    &=&  \lozenge \neg x \vee c \\
                     & = & \neg^{\lozenge} x \vee c      & = & h(\neg^{\lozenge} x).
\end{array}
\]
We can conclude that $h$ is an isomorphism, as desired.
\end{proof}

\begin{corollary} \label{3construc}
Let $T \in {\bf{SKANc}}$. Then the centered Stone KAN-algebras $K(T^{\theta})$, $K(T^{\lozenge})$ and $K(C(T))$ are isomorphic.
\end{corollary}

We denote by ${\bf{Stone}}$ the category whose objects are Stone algebras and by ${\bf{SKANc}}$ the category whose objects are Stone KANc-algebras. In both cases, the morphisms are the corresponding algebra homomorphisms. Furthermore, if $A$ and $B$ are two Stone algebras and $h \colon A \to B$ is a morphism in ${\bf{Stone}}$, then it is easy to see that the map $K(h) \colon K(A) \to K(B)$ defined by $K(h)(x,y) = (h(x), h(y))$ is a morphism in ${\bf{SKANc}}$. 
It is evident that these assignments establish a functor $K$ from ${\bf{Stone}}$ to ${\bf{SKANc}}$. Moreover, if $T,S \in {\bf{SKANc}}$ and $f \colon T \to S$ is a homomorphism of centered Stone KAN-algebras, then $C(f) \colon C(T) \to C(S)$ defined by $C(f)(x) = f(x)$ is a homomorphism of Stone algebras. 
By Lemma \ref{lemaant}, it is now clear that the assignments $T \mapsto C(T)$ and $f \mapsto C(f)$ determine a functor $C$ from ${\bf{SKANc}}$ to ${\bf{Stone}}$.

\begin{lemma} \label{lemaalpha} 
Let $A$ be a Stone algebra. Then the map $\alpha \colon  A \to C(K(A))$ defined by $\alpha(x) := (x,0)$ is an isomorphism in ${\bf Stone}$. 
\end{lemma}
\begin{proof} 
We only prove that the map $\alpha$ commutes with the unary operator $^{\ast}$. If $x \in A$, then 
\[
\neg^{c} \alpha(x) = \neg \alpha(x)\vee c=(x^{\ast},x)\vee (0,0)=(x^{\ast},0)=\alpha(x^{\ast}),
\]
as claimed.
\end{proof}

\begin{corollary} 
Let $A$ be a Stone algebra. Then the Stone algebras $A$ and $K(A)^{\lozenge}$ are isomorphic. 
\end{corollary}
\begin{proof}
This follows from Theorem \ref{C(T)=H} and Lemma \ref{lemaalpha}.
\end{proof}

\begin{lemma} \label{isobeta} 
Let $T\in {\bf SKANc}$. Then the map $\beta \colon T \to K(C(T))$ defined by $\beta(x) := (x \vee c, \sim x \vee c)$ is an isomorphism in ${\bf{SKANc}}$.    
\end{lemma}
\begin{proof} 
Let $x \in T$. Then $\beta(x) \in K(C(T))$. Indeed, $x \vee c, \sim x \vee c \in C(T)$ and by Lemmas \ref{auxiliar2} and \ref{t} we have 
\[
(x \vee c) \wedge (\sim x \vee c) = (x \wedge \sim x) \vee c = \lozenge(x \wedge \sim x) \vee c = c.
\]
Now, we show that $\beta$ commutes with the unary operator $\neg$.  By Lemma \ref{auxiliar2}, it follows that 
\[
\beta(\neg x)  =  (\neg x \vee c, \lozenge x \vee c) = (\neg x \vee c, x \vee c) = (\neg^{c}(x \vee c), x \vee c), 
\]
i.e., $\beta(\neg x) = \neg \beta(x)$.    
\end{proof}

According to Lemmas \ref{lemaalpha} and \ref{isobeta} the following theorem is deduced.

\begin{theorem} \label{equivalence theorem}
The functors $K$ and $C$ establish a categorical equivalence between ${\bf{Stone}}$ and ${\bf{SKANc}}$ with natural isomorphisms $\alpha$ and $\beta$.
\end{theorem}

\section{Monteiro's construction}\label{sec4}

Given a Nelson algebra, it is always possible to obtain a centered Nelson algebra through a process known as Monteiro's construction, which, notably, was never published. Inspired by this construction, we apply a similar method to obtain a centered Stone KAN-algebra from a Stone KAN-algebra. \
Although our method draws inspiration from Monteiro's approach, it is more general. Consequently, Monteiro's contributions and their corresponding proofs were not applicable to our work.

Let $T \in {\textbf{KAN}}$. Consider the set 
\[
M(T) := \{ (x,y) \in T^{\lozenge} \times T^{\square} \colon x \leq y \}
\]
and the following operators on $M(T)$ defined by:
\begin{eqnarray*}
(a,b) \cup (d,e)   & := & (a \vee d, \square (b \vee e)), \\
(a,b) \cap (d,e)   & := & (\lozenge (a \wedge d), b \wedge e), \\
\approx (a,b)      & := & (\sim b, \sim a), \\
\circledast (a,b)  & := & (\lozenge \neg a, \sim a), \\
0_M                        & := & (0,0), \\
1_M                        & := & (1,1), \\
c_M                        & := & (0,1). \\
\end{eqnarray*}
It follows by Propositions \ref{Prop_3}, \ref{Prop1} and Lemma \ref{auxiliar2} that $(a,b) \cap (d,e), (a,b) \cup (d,e), 0_{M}, 1_{M}, c_{M} \in M(T)$, for all $(a,b), (d,e) \in M(T)$.

\begin{lemma}
Let $T \in {\bf{KAN}}$. If $(a,b) \in M(T)$, then $\approx (a,b), \circledast (a,b) \in M(T)$.
\end{lemma}
\begin{proof}
If $(a,b)\in M(T)$, then $\lozenge a = a$, $\square b = b$ and $a \leq b$. First we prove that $\approx (a,b) \in M(T)$. Since $a \leq b$, we have $\sim b \leq \sim a$. As $b \in T^{\square}$, then $\sim b \in T^{\lozenge}$. On the other hand, $a \in T^{\lozenge}$ implies $\neg a = \sim a$ and $\square \sim a = \neg \sim \sim a = \neg a = \sim a$. Thus, $\sim a \in T^{\square}$ and $\approx (a,b) \in M(T)$.

Now, we see $\circledast (a,b) \in M(T)$. It is clear that $\lozenge \neg a \in T^{\lozenge}$. Since $\square \sim a = \sim a$, we have $\sim a \in T^{\square}$. So, by Proposition \ref{Prop1}, it follows $\lozenge \neg a \leq \neg a = \sim a$, i.e., $\lozenge \neg a \leq \sim a$ and $\circledast (a,b) \in M(T)$.
\end{proof}

We can consider the structure $\langle M(T), \cup, \cap, \approx, \circledast, c_{M}, 0_{M}, 1_{M} \rangle$. Then the order $\subseteq$ on $M(T)$ is given by $(a,b) \subseteq (d,e)$ if and only of $a \leq d$ and $b \leq e$.

\begin{theorem} \label{teoM}
Let $T \in {\bf{KAN}}$. Then $M(T) \in {\bf{KANc}}$.
\end{theorem}
\begin{proof}
Let $(a,b), (d,e) \in M(T)$. Then $\lozenge a = a$, $\square b = b$, $a \leq b$, $\lozenge d = d$, $\square e = e$ and $d \leq e$. We prove that $M(T)$ is a distributive lattice.
First we see 
\[
(a,b) \cap \left[ ((a,b) \cup (d,e)) \right] = (a,b),
\]
i.e., $\lozenge(a \wedge (a \vee d)) = a$ and $b \wedge \square(b \vee e) = b$. Indeed, since $a \wedge (a \vee d) = a$ is valid in $T$, then we have $\lozenge(a \wedge (a \vee d)) = \lozenge a = a$. On the other hand, it follows $b \leq b \vee e \leq \square(b \vee e)$ and $b \wedge \square(b \vee e) = b$.

Let $(f,g) \in M(T)$. So, $\lozenge f = f$, $\square g = g$ and $f \leq g$. We prove 
\[
(a,b) \cap [(d,e) \cup (f,g)] = [(f,g) \cap (a,b)] \cup [(d,e) \cap (a,b)].
\]
It will be enough to try $\lozenge(a \wedge (d \vee f)) = \lozenge(f \wedge a) \vee \lozenge(d \wedge a)$ and $b \wedge \square(e \vee g) = \square((g \wedge b) \vee (e \wedge b))$. Since $a \wedge (d \vee f) = (f \wedge a) \vee (d \wedge a)$ and $b \wedge (e \vee g) = (g \wedge b) \vee (e \wedge b)$ are valid in $T$, then by Proposition \ref{Prop_3} we have $\lozenge(a \wedge (d \vee f)) = \lozenge(f \wedge a) \vee \lozenge(d \wedge a)$ and $b \wedge \square(e \vee g) = \square b \wedge \square(e \vee g) = \square((g \wedge b) \vee (e \wedge b))$. So, by Sholander's conditions given in \cite{Sholander}, the structure $\langle M(T), \cup, \cap, (0,0), (1,1) \rangle$ is a bounded distributive lattice. 

Now, we prove that $\langle M(T), \cup, \cap, \approx, (0,0), (1,1) \rangle$ is a Kleene algebra. Let $(a,b), (d,e) \in M(T)$.
\begin{itemize}
\item[(K1):] It is immediate that $\approx \approx (a,b) = (a, b)$.

\item[(K2):] Since $\sim \square x = \lozenge \sim x$, we have 
\[
\begin{array}{lllll}
\approx [(a,b) \cup (d,e)] & = & \approx (a \vee d, \square(b \vee e)) \\
                                      & = & (\sim \square(b \vee e), \sim (a \vee d)) \\
                                      & = & (\lozenge (\sim b \wedge \sim e), \sim a \wedge \sim d)) \\
                                      & = & (\sim b, \sim a) \cap (\sim e, \sim d) \\
                                       & = & \approx (a, b) \cap \approx (d, e).\\
\end{array}   
\]
Then $\approx [(a,b) \cup (d,e)] = \approx (a, b) \cap \approx (d, e)$.

\item[(K3):] As $a \leq b$, by Lemma \ref{t} we have $\lozenge(a \wedge \sim b) \leq \lozenge(b \wedge \sim b) = 0$ and $\lozenge(a \wedge \sim b) = 0$. Moreover, by $d \leq e$ and Lemma \ref{t}, it follows that $1 = \square(d \vee \sim d) \leq \square(e \vee \sim d)$. So, $\square(e \vee \sim d) = 1$. Then 
\[
\begin{array}{lllll}
(a,b) \cap \approx (a,b) & =  (a,b) \cap (\sim b, \sim a) \\
                                  & = (\lozenge (a \wedge \sim b), b \wedge \sim a) \\
                                  & \subseteq  (d \vee \sim e, \square(e \vee \sim d)) \\
                                  & = (d,e) \cup (\sim e, \sim d) \\
                                  & = (d,e) \cup \approx (d,e). \\
\end{array}
\]    
\end{itemize}

Finally, let us prove axioms (N1)-(N5) of Definition \ref{defKAN}.
\begin{itemize}
\item[(N1):] Note that $\neg (a \wedge \lozenge \neg \lozenge(a \wedge d)) = \neg (a \wedge \lozenge \neg d)$. Indeed, by (N5), Proposition \ref{Prop1} and (N1) we have 
\[
\begin{array}{lllll}
\neg (a \wedge \lozenge \neg \lozenge(a \wedge d)) & = & \neg (a \wedge \sim \neg \neg \lozenge(a \wedge d)) \\
                                                    & = & \neg (a \wedge \neg \lozenge(a \wedge d)) \\
                                                    & = & \neg (a \wedge \neg (a \wedge d)) \\
                                                    & = & \neg (a \wedge \neg d) \\
                                                    & = & \neg (a \wedge \sim \neg \neg d) \\
                                                    & = & \neg (a \wedge \lozenge \neg d). \\
\end{array}    
\]
Then,  
\[
\begin{array}{llllll}
\circledast ((a,b) \cap \circledast ((a,b) \cap (d,e))) & = & \circledast ((a,b) \cap \circledast (\lozenge(a \wedge d), b \wedge e)) \\
                                                                               & = & \circledast ((a,b) \cap (\lozenge \neg \lozenge(a \wedge d), \sim \lozenge(a \wedge d))) \\
                                                                               & = & \circledast (\lozenge(a \wedge \lozenge \neg \lozenge(a \wedge d)), b \wedge \sim \lozenge(a \wedge d))\\
                                                                              & = & (\lozenge \neg \lozenge(a \wedge \lozenge \neg \lozenge(a \wedge d)), \sim \lozenge(a \wedge \lozenge \neg \lozenge(a \wedge d))) \\
                                                    & = & (\lozenge \neg \lozenge (a \wedge \lozenge \neg d), \sim \lozenge(a \wedge \lozenge \neg d)) \\
                                                    & = & \circledast (\lozenge(a \wedge \lozenge \neg d), b \wedge \sim d) \\
                                                    & = & \circledast ((a,b) \cap (\lozenge \neg d, \sim d)) \\
                                                    & = & \circledast((a,b) \cap \circledast (d,e)). \\                                                                                                        
\end{array}    
\]

\item[(N2):] By (N2) and Proposition \ref{Prop_3} we have 
\begin{align*}
\circledast ((a,b) \cup (d,e)) & =  \circledast (a \vee d, \square(b \vee e)) \\
                                                    & = (\lozenge \neg (a \vee d), \sim (a \vee d)) \\
                                                    & = (\lozenge (\neg a \wedge \neg d), \sim a \wedge \sim d) \\
                                                    & = (\lozenge( \lozenge \neg a \wedge \lozenge \neg d), \sim a \wedge \sim d) \\
                                                    & = (\lozenge \neg a, \sim a) \cap (\lozenge \neg d, \sim d) \\
                                                    & = \circledast (a,b) \cap \circledast (d,e). \\
\end{align*}    

\item[(N3):] First we prove that $\lozenge(a \wedge \lozenge \neg a) = \lozenge(a \wedge \sim b)$. Since $a \leq b$, by (N5), (N1) and (N3) we have 
\[
\begin{array}{lllll}
\neg (a \wedge \sim \neg \neg a) & = & \neg (a \wedge \neg a) \\
                                                    & = & \neg (a \wedge \neg (a \wedge b)) \\
                                                    & = & \neg (a \wedge \neg b) \\
                                                    & = & \neg (a \wedge b \wedge \neg b) \\
                                                    & = & \neg (a \wedge b \wedge \sim b) \\
                                                    & = & \neg (a \wedge \sim b). \\
\end{array}    
\]
So, $\neg (a \wedge \sim \neg \neg a) = \neg (a \wedge \sim b)$ and $\lozenge(a \wedge \lozenge \neg a) = \lozenge(a \wedge \sim b)$. Then
\[
\begin{array} {lllll}
(a,b) \cap \approx (a,b) & = & (a,b) \cap (\sim b, \sim a) \\
                                     & = & (\lozenge(a \wedge \sim b), b \wedge \sim a) \\
                                     & = & (\lozenge(a \wedge \lozenge \neg a), b \wedge \sim a) \\
                                     & = & (a,b) \cap (\lozenge \neg a, \sim a) \\
                                     & = & (a,b) \cap \circledast (a,b). \\
\end{array}    
\]

\item[(N4):] Since $\lozenge a = a$ and $a \leq b$, it follows $\neg a = \sim a$ and $\sim b \leq \neg a$. So, we have $\neg \neg a \leq \neg \sim b = \square b = b$ and $\sim b \leq \sim \neg \neg a$. Then 
\[
\begin{array}{lllll}
\approx (a,b) & = & (\sim b, \sim a) \\
                     & \subseteq & (\sim \neg \neg a, \sim a) \\
                     & = & (\lozenge \neg a, \sim a) \\
                     & = & \circledast (a,b). \\
\end{array}
\]    

\item [(N5):] We have 
\[
\begin{array}{lllll}
\circledast ((a,b) \cap (d,e)) & = & \circledast (\lozenge(a \wedge d), b \wedge e) \\
                                            & = & (\lozenge \neg \lozenge(a \wedge d), \sim \lozenge(a \wedge d)) \\
                                            & = & \circledast (\lozenge(a \wedge d), \sim \lozenge \neg a \wedge e) \\
                                            & = & \circledast ((a, \sim \lozenge \neg a) \cap (d,e)) \\
                                            & = & \circledast ( \approx (\lozenge \neg a, \sim a) \cap (d,e)) \\
                                           & = & \circledast ( (\approx \circledast (a,b))  \cap (d,e)). \\
\end{array}   
\] 
\end{itemize}

It is easy to see that $c_{M} = (0,1)$ is the center of $M(T)$. Then $M(T)$ is a KANc-algebra.
\end{proof}

\begin{remark}
Let ${\bf{KAN}}$ denote the category of KAN-algebras and ${\bf{KANc}}$ the category of centered KAN-algebras. In both cases, morphisms are the corresponding algebra homomorphisms. Notably, if $T_1$ and $T_2$ are centered KAN-algebras and $f \colon T_1 \to T_2$ is a morphism in ${\bf{KAN}}$, then $f$ necessarily preserves the center. Indeed, we have $f(c) = f(\sim c) = \sim f(c)$ and thus, by the uniqueness of the center, $f(c) = c$. Therefore, ${\bf{KANc}}$ is a full subcategory of ${\bf{KAN}}$. Additionally, if $f \colon T_1 \to T_2$ is a morphism in ${\bf{KAN}}$, then $M(f) \colon M(T_1) \to M(T_2)$ defined by $M(f)(x,y) = (f(x), f(y))$ for all $(x, y) \in M(T_1)$, is a morphism in ${\bf{KANc}}$. Moreover, by Theorem \ref{teoM}, $M$ is a functor from the category ${\bf{KAN}}$ to ${\bf{KANc}}$.
\end{remark}

In the following result we prove a necessary and sufficient conditions for $M(T)$ to be a Stone KAN-algebra.

\begin{lemma} \label{lemaM}
Let $T \in {\bf{KAN}}$. Then $T$ is a Stone KAN-algebra if and only $M(T)$ is a Stone KAN-algebra. 
\end{lemma}
\begin{proof}
Let $(a,b)\in M(T)$. Then $\lozenge a=a$, $\square b=b$ and $a \leq b$. So, $\sim a = \neg a$. By Proposition \ref{Prop1} we have $\neg \neg a = \neg \lozenge \neg a$. Hence, by Proposition \ref{Prop1} and \ref{Prop_3}, it follows 
\[
\begin{array}{llllll}
\circledast(a,b) \cup \circledast\circledast(a,b) &=& (\lozenge \neg a, \sim a) \cup (\lozenge \neg \lozenge \neg a, \sim \lozenge \neg a) \\
                                                                         &=& (\lozenge \neg a \vee \lozenge \neg \lozenge \neg a, \square(\sim a \vee \sim \lozenge \neg a)) \\
                                                                         &=& (\lozenge(\neg a \vee \neg \lozenge \neg a), \square (\neg a \vee \neg \neg a))\\
                                                                         &=& (\lozenge(\neg a \vee \neg \neg a), \square 1) \\
                                                                         &=& (\lozenge 1, \square 1) \\
                                                                         &=& (1, 1).
\end{array}
\]
Then $M(T)$ is Stone KAN-algebra. Now, suppose that $T$ satisfies the equation \ref{stoneKAN}. Let $a \in T$. Then $(\lozenge a, \square a) \in M(T)$ and by Proposition \ref{Prop1} and \ref{Prop_3},
\[
\begin{array}{lllll}
(1,1)  &=& \circledast(\lozenge a,\square a) \cup \circledast\circledast(\lozenge a,\square a) \\
         &=& (\lozenge \neg \lozenge a, \sim \lozenge a) \cup (\lozenge \neg \lozenge \neg \lozenge a, \sim \lozenge \neg \lozenge a) \\
         &=& (\lozenge \neg \lozenge a \vee \lozenge \neg \lozenge \neg \lozenge a, \square (\sim \lozenge a \vee \sim \lozenge \neg a)) \\
         &=& (\lozenge (\neg \lozenge a \vee \neg \lozenge \neg \lozenge a), \square (\sim \lozenge a \vee \sim \lozenge \neg a)) \\
         &=& (\lozenge (\neg a \vee \neg \neg a), \square (\neg a \vee \neg \neg a)).
\end{array}
\]
So, it follows that $\lozenge (\neg a \vee \neg \neg a) = \lozenge 1$ and $\square (\neg a \vee \neg \neg a) = \square 1$. Then, by Moisil's determination principle of Lemma \ref{MDP}, we have $\neg a \vee \neg \neg a = 1$ and $T$ is Stone KAN-algebra.
\end{proof}

\begin{corollary}\label{M(T)}
Let $T \in {\bf{SKAN}}$. Then $M(T) \in {\bf{SKANc}}$.
\end{corollary}
\begin{proof} 
It follows from Theorem \ref{teoM} and Lemma \ref{lemaM}.
\end{proof}

\begin{theorem}
Let $T \in {\bf{SKAN}}$. Then $T$ is isomorphic to a subalgebra of some centered Stone KAN-algebra.
\end{theorem}
\begin{proof}
By Corollary \ref{M(T)}, the structure $M(T)$ is a Stone KANc-algebra. Consider the map $\delta \colon T \to M(T)$ defined by $\delta(x) := (\lozenge x, \square x)$. Using Propositions \ref{Prop1} and \ref{Prop_3}, we can show that $\delta$ is well-defined. Additionally, the equalities $\delta(0) = (0,0)$ and $\delta(1) = (1,1)$ are straightforward. Furthermore, by Lemma \ref{lemacentro}, we obtain $\delta(c) = (0,1)$. Furthermore, by Proposition \ref{Prop_3}, we have 
\[
\begin{array}{lllll}
 \delta(x) \cup \delta(y) & = & (\lozenge x, \square x) \cup (\lozenge y, \square y) \\
                                    & = & (\lozenge x \vee \lozenge y, \square (\square x \vee \square y)) \\
                                    & = & (\lozenge (x \vee y), \square (x \vee y)) \\
                                    & = & \delta(x \vee y).
\end{array}
\]
Similarly, $\delta(x \wedge y) = \delta(x) \cap \delta(y)$. Using notations \ref{Delta} and \ref{nabla}, we obtain $\delta(\sim x) = \approx \delta(x)$. Finally, by Proposition \ref{Prop1},
\[
\circledast \delta(x) = (\lozenge \neg \lozenge x, \sim \lozenge x) = (\lozenge \neg x, \square \neg x) = \delta(\neg x).
\]
Thus, $\delta$ is a homomorphism. The injectivity of $\delta$ follows from Lemma \ref{MDP}.
\end{proof}

The following result illustrates the relationship between Kalman’s and Monteiro’s constructions in the context of Stone KAN-algebras.

\begin{theorem} \label{teoisoMK}
Let $T \in {\bf{SKAN}}$. Then the centered Stone KAN-algebras $M(T)$ and $K(T^{\lozenge}) $ are isomorphic.
\end{theorem}
\begin{proof}
Consider the map $t \colon M(T) \to K(T^{\lozenge})$ given by $t(x, y) := (x, \sim y)$. Let us prove that $t$ is an isomorphism. We need to show that $t$ is well-defined. If $(a,b) \in M(T)$, then $\lozenge a = a$, $\square b = b$ and $a \leq b$. It follows that $\sim b \in T^{\lozenge}$. Since $a \wedge \sim b \leq b \wedge \sim b$, by Lemmas \ref{Prop2} and \ref{t} we deduce that $\lozenge(a \wedge \sim b) \leq \lozenge(b \wedge \sim b) = 0$, i.e., $\lozenge(a \wedge \sim b) = 0$. Hence, $a \wedge^{\lozenge} \sim b = \lozenge(a \wedge \sim b) = 0$ and $t(x,y) = (x, \sim y) \in K(T^{\lozenge})$.

Next, we prove that $t$ is surjective. Let $(a,b) \in K(T^{\lozenge})$. Then $a \in T^{\lozenge}$, $b \in T^{\lozenge}$ and $a \wedge^\lozenge b = \lozenge(a \wedge b) = 0$. By Lemma \ref{lemma2.23} it follows $a \leq \sim b$. So, since $b \in T^{\lozenge}$, we have $\sim b \in T^{\square}$. Thus, $(a, \sim b) \in M(T)$ and $t(a, \sim b) = (a, b)$. The injectivity of $t$ is straightforward.

Finally, we see that $t$ is a homomorphism. Let $(a, b), (d, e) \in M(T)$. Then, by Proposition \ref{Prop_3}, we have
\[
\begin{array}{lllll}
t(\circledast (a,b)) & = & t(\lozenge \neg^{\lozenge} a, \sim a) \\
                             & = & (\lozenge \lozenge \neg a, a) \\
                             & = &  (\lozenge \neg a, a) \\
                             & = & (\neg^{\lozenge} a, a) \\
                             & = & \neg t(a,b) 
\end{array}
\]
and
\[
\begin{array}{llllll}
t((a,b) \cap (d,e)) & = & t(\lozenge(a \wedge d), b \wedge e) \\
                            & = &  (\lozenge(a \wedge d), \sim (b \wedge e)) \\
                            & = & (a \wedge^{\lozenge} d, \sim b \vee \sim e)) \\
                            & = & t(a,b) \wedge t(d,e).
\end{array}
\]
Therefore, $t$ preserves the operations $\circledast$ and $\cap$. The remainder of the proof, which involves additional details, is left for the reader to verify.
\end{proof}

By Lemma \ref{isobeta}, Corollary \ref{3construc} and Theorem \ref{teoisoMK} we have the following result.

\begin{corollary} 
Let $T \in {\bf{SKANc}}$. Then the centered Stone KAN-algebras $T$ and $M(T)$ are isomorphic.
\end{corollary}

\begin{example} 
Consider the Stone KAN-algebra from Example \ref{SKANalgebra}. Then we have the Stone KANc-algebra $M(T)$ given by
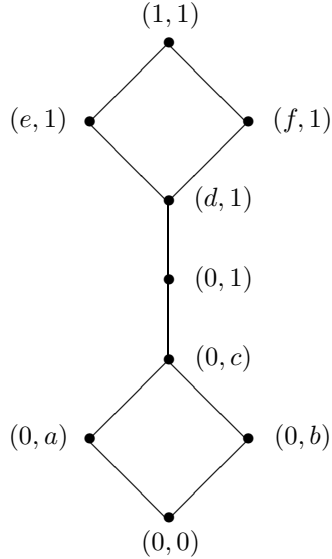
\vspace{0.5cm}
\begin{center}
\hspace{0.25cm}
\put(00,00){\makebox(1,1){$\bullet$}}
\put(-30,30){\makebox(1,1){$\bullet$}}
\put(30,30){\makebox(1,1){$\bullet$}}
\put(00,60){\makebox(1,1){$\bullet$}}
\put(00,90){\makebox(1,1){$\bullet$}}
\put(00,120){\makebox(1,1){$\bullet$}}
\put(-30,150){\makebox(1,1){$\bullet$}}
\put(30,150){\makebox(1,1){$\bullet$}}
\put(00,180){\makebox(1,1){$\bullet$}}
\put(00,00){\line(1,1){30}}
\put(00,00){\line(-1,1){30}}
\put(-30,30){\line(1,1){30}}
\put(30,30){\line(-1,1){30}}
\put(00,60){\line(0,0){30}}
\put(00,90){\line(0,0){30}}
\put(00,120){\line(1,1){30}}
\put(00,120){\line(-1,1){30}}
\put(-30,150){\line(1,1){30}}
\put(30,150){\line(-1,1){30}}
\put(00,-10){\makebox(2,2){$ (0,0)$}}
\put(-50,30){\makebox(2,2){$ (0,a)$}} 
\put(50,30){\makebox(2,2){$ (0,b)$}} 
\put(20,60){\makebox(2,2){$ (0,c)$}} 
\put(20,90){\makebox(2,2){$(0,1)$}}
\put(20,120){\makebox(2,2){$ (d,1)$}}
\put(-50,150){\makebox(2,2){$ (e,1)$}} 
\put(50,150){\makebox(2,2){$ (f,1)$}} 
\put(00,190){\makebox(2,2){$ (1,1)$}}
\end{center}
\vspace{0.35cm}
Then the Stone KANc-algebras $M(T)$ and $K(T^{\lozenge})$ from Example \ref{ejKalman} are isomorphic.
\end{example}

\section{Conclusions and future work}

In this paper we introduced the variety of Stone-Kleene algebras with intuitionistic negation, or Stone KAN-algebras, and established a categorical equivalence between the category of Stone algebras and the category of centered Stone KAN-algebras. By employing Kalman's construction, we provided a method to construct a centered Stone KAN-algebra from a given Stone algebra. Additionally, drawing inspiration from Monteiro's construction, we given a construction for centered Stone KAN-algebras showing the relationship between the two constructions.

In future work, we plan to extend these results to the context of KAN-algebras with additional operations. This includes the tense operators defined on KAN-algebras given in \cite{Gustavo-Maia} and the weak quantifiers introduced in \cite{Conrado-Miguel-Hernan}.


\end{document}